\documentclass[a4paper,12pt]{article}

\usepackage{amsmath,amsthm,array}
\usepackage{amssymb}
\usepackage{amsrefs}
\usepackage{mathtools}
\usepackage{enumerate}
\usepackage{tikz-cd}
\usepackage[margin=2cm]{geometry}
 
\usepackage[utf8]{inputenc}
\usepackage{tgtermes}
\usepackage{mathptmx}
\usepackage[T1]{fontenc}
\usepackage{stackrel}
\usepackage{pict2e,picture}
\usepackage{xcolor}
 \usepackage{relsize}

\usepackage{comment}

\usepackage{hyperref}
\hypersetup{colorlinks=true}

\usepackage{derivative}

\usepackage{fixme}
\fxsetup{status=draft, layout=inline, author=, theme=color}

\newtheorem{theorem}{Theorem}[section]
\newtheorem{lemma}[theorem]{Lemma}
\newtheorem{proposition}[theorem]{Proposition}
\newtheorem{corollary}[theorem]{Corollary}

\theoremstyle{definition}
\newtheorem{example}[theorem]{Example}

\newtheorem{remark}[theorem]{Remark}

\newcommand{\An}{A_n}

\newcommand{\ZZ}{\mathbb{Z}}

\newcommand{\PGL}{\operatorname{PGL}}

\newcommand{\End}{\operatorname{End}}

\newcommand{\Mat}{\mathrm{Mat}}

\newcommand{\Tr}{\operatorname{tr}}

\newcommand{\ad}{\operatorname{ad}}

\newcommand{\leftbimod}[3]{\vphantom{#1}^{#2}{\kern-#3pt #1}}

\usepackage{fixme}
\fxsetup{status=draft, layout=inline, author=, theme=color}

\numberwithin{equation}{subsection}

\title{Lifts of endomorphisms of Weyl algebras modulo $p^2$}
\date{}

\author{Niels Lauritzen and Jesper Funch Thomsen} 
\begin{document}
\maketitle 

\begin{abstract}
Let $\varphi$ denote a $k$-algebra endomorphism of the $n$-th Weyl algebra 
$\An(k)$ over a perfect field $k$ of positive characteristic $p$. 
We prove that $\varphi$ can be lifted to an endomorphism of the Weyl 
algebra $\An(W_2(k))$ over the Witt vectors $W_2(k)$ of length two over $k$ 
if and only if $\varphi$ induces a Poisson morphism of the center of 
$\An(k)$. Furthermore, we improve a result of Tsuchimoto, which enables us to conclude that these equivalent statements hold at least when $\deg(\varphi) < p$. %the degree of $\varphi$ is less than $p$. 
In particular, we conclude that $\varphi$ is injective if $\deg(\varphi) < p$.
 \end{abstract}

\section*{Introduction}

Let $k$ be a perfect field of positive characteristic $p>0$, and let $\An(k)$ denote the $n$-th Weyl algebra over $k$, generated by elements
$$
z_1, z_2, \dots, z_{2n},
$$
subject to the relations
$$
[z_i, z_j] := \omega_{i,j}, \qquad 1 \leq i,j \leq 2n,
$$
where $\omega_{i,j}$ is the $(i,j)$-th entry of the matrix
$$
\omega =
\begin{pmatrix}
0 & -I_n \\
I_n & 0
\end{pmatrix} \in \Mat_{2n}(k),
$$
and $I_n \in \Mat_n(k)$ denotes the $n \times n$ identity matrix.  

In positive characteristic, the Weyl algebra is an Azumaya algebra with center $Z$ equal to a polynomial ring over $k$ in the $2n$ variables
$$
x_i := z_i^p, \quad i=1,\dots,2n.
$$
We consider the center $Z$ as a Poisson algebra with Poisson bracket 
defined by
$$
\{ x_i, x_j \} = -\omega_{i,j}, \qquad 1 \leq i,j \leq 2n.
$$

Let $\varphi$ be a $k$-algebra endomorphism of $\An(k)$. 
It is known that $\varphi(Z) \subseteq Z$, and one may therefore ask
under what conditions the induced endomorphism $\varphi_Z$ of $Z$
is a Poisson morphism, i.e.
$$
\varphi_Z(\{f,g\}) = \{\varphi_Z(f), \varphi_Z(g)\},
\qquad f,g \in Z.
$$
One typically expects this to be the case; however, counterexamples exist, even when 
$\varphi$ is an automorphism (see Example \ref{exnonpoisson}).
Nevertheless, some partial results are known.
Tsuchimoto \cite[Cor. 3.3]{Tsuchimoto2005Endomorphisms} proved that if the degree of $\varphi$ 
(to be defined below) is less than $p/2$, then $\varphi_Z$ is a 
Poisson morphism. In a slightly different direction, Belov-Kanel and Kontsevich 
 observed \cite{KanKon07} that if $\varphi$ can be lifted to an endomorphism 
of the Weyl algebra over a ring of characteristic $p^2$, then 
$\varphi_Z$ is  a Poisson morphism.  
As a side remark, it has been conjectured by Belov-Kanel and Kontsevich 
\cite{BelovKontsevich2005}
that if $k$ is a field of characteristic zero, there is expected to exist a 
canonical isomorphism between the automorphism group of the Weyl algebra 
$\An(k)$ and the group of symplectomorphisms of affine $2n$-space over $k$.

In this paper, we study the question of when a $k$-algebra endomorphism 
$\varphi$ can be lifted to an endomorphism of the Weyl algebra over the 
Witt vectors $W_2(k)$ of length two over $k$. We show that such a lifting 
exists precisely when the induced endomorphism $\varphi_Z$ of the center 
$Z$ is a Poisson morphism. 
Combining this with the result of Tsuchimoto, we see that $\varphi$ can 
be lifted to the Witt vectors of length two if the degree of $\varphi$ 
is less than $p/2$. Along the way, we also note that Tsuchimoto's result 
can be improved to include   the case where $\varphi$ has degree 
less than $p$, although it holds in even greater generality 
(see Thm. \ref{etalecenter}). We furthermore give an if-and-only-if criterion 
for ${\varphi}_Z$ to be a Poisson morphism, determined by the solutions to 
certain differential equations introduced by Tsuchimoto (see Prop. \ref{diffeqpois}).

In the last section (\S \ref{lastsect}) of the paper we give a few applications of this 
improved degree bound and restriction to the center. 
Injectivity of $\varphi$ always holds for $n=1$. Bavula conjectured it to hold in general, but Makar-Limanov gave an example of a non-injective endomorphism for $n=2$. 
Under the assumption that $\varphi <p$ (and, more generally, under weaker assumptions), 
 we prove that $\varphi$ is injective.

We also show that $\varphi$ is flat and an automorphism if it is birational
$\deg(\varphi)< p$. These two results were previously only known for fields of 
characteristic zero.
We refer to \S \ref{lastsect} for precise statements and references.

\section{Notation}
\label{notation}

Let $k$ be a perfect field of positive characteristic $p>0$, and 
let $\An(k)$ denote the $n$-th Weyl algebra over $k$. We view $\An(k)$ as generated by elements
$
z_1, z_2, \dots, z_{2n},
$ 
subject to the relations
$$
[z_i, z_j] := \omega_{i,j}, \qquad 1 \le i,j \le 2n,
$$
where $\omega_{i,j}$ is the $(i,j)$-th entry of
$$
\omega =
\begin{pmatrix}
0 & -I_n \\
I_n & 0
\end{pmatrix} \in \Mat_{2n}(k),
$$
and $I_n \in \Mat_n(k)$ denotes the $n \times n$ identity matrix.  

As a $k$-vector space, $\An(k)$ has basis
$$
z_1^{i_1} z_2^{i_2} \cdots z_{2n}^{i_{2n}}, \quad i_1,\dots,i_{2n} \in \ZZ_{\ge 0}.
$$
We will use multi-index notation and write such an element as ${\bf z}^{\bf i}$, where ${\bf i} = (i_1, \dots, i_{2n})$. 
The total degree of a basis element is defined by
$$
\deg({\bf z}^{\bf i}) := i_1 + \dots + i_{2n},
$$
and for a general nonzero element
$$
f = \sum_{\bf i} c_{\bf i} \, {\bf z}^{\bf i} \in \An(k), \quad c_{\bf i} \in k,
$$
we set
$$
\deg(f) := \max\{\deg({\bf z}^{\bf i}) : c_{\bf i} \neq 0\}.
$$

The center of $\An(k)$ is denoted by $Z$ and is a polynomial ring over $k$ in $2n$ variables, generated by
$$
x_i := z_i^p, \qquad i=1,\dots,2n.
$$
As a module over $Z$, the Weyl algebra is free of rank $p^{2n}$ with basis
$$
z_1^{i_1} z_2^{i_2} \cdots z_{2n}^{i_{2n}}, \quad 0 \le i_1,\dots,i_{2n} < p.
$$

Throughout, we let $\varphi$ denote a $k$-algebra endomorphism of $\An(k)$. 
We define the degree $\deg(\varphi)$ of $\varphi$ to be the maximal degree of the elements $\varphi(z_i)$, $i=1,\dots,2n$.
It is known (\cite[Cor. 2]{Tsuchimoto2003Preliminaries}) that $\varphi$ defines a $Z$-module basis
$$
\varphi(z_1)^{i_1} \varphi(z_2)^{i_2} \cdots \varphi(z_{2n})^{i_{2n}}, 
\quad 0 \le i_1,\dots,i_{2n} < p.
$$
In particular, any element of $\varphi(Z)$ is central, as it commutes with all basis elements. 
Hence 
$
\varphi(Z) \subseteq Z,
$
and thus $\varphi$, by restriction, induces a $k$-algebra endomorphism $\varphi_Z$ of $Z$.

Throughout this article, the notation $\delta_{i,j}$ is used for Kronecker's delta. 
Moreover, if $R$ is a ring and $x \in R$, we write $\ad(x)$ for the map 
from $R$ to itself given by the commutator with $x$.

\subsection{Witt vecors}

The ring of Witt vectors of length two over $k$, denoted $W_2(k)$,
is the set
$$
W_2(k) = \{(a_1,a_2) \mid a_1,a_2 \in k\},
$$
equipped with addition and multiplication given by
$$
(a_1,a_2) + (b_1,b_2)
=
\Bigl(
a_1+b_1,\;
a_2+b_2+\frac{1}{p}\bigl(a_1^p+b_1^p-(a_1+b_1)^p\bigr)
\Bigr),
$$
and
$$
(a_1,a_2)\cdot(b_1,b_2)
=
\Bigl(
a_1 b_1,\;
a_1^p b_2 + b_1^p a_2
\Bigr).
$$
Here the expression
$$
\frac{1}{p}\bigl(a_1^p+b_1^p-(a_1+b_1)^p\bigr)
$$
is interpreted as the reduction modulo $p$ of an integral polynomial,
which is well-defined since the numerator is divisible by $p$ in
$\mathbb{Z}[a_1,b_1]$.

The additive and multiplicative identities of $W_2(k)$ are $(0,0)$ and
$(1,0)$, respectively. Moreover,
$$
p\cdot(a_1,a_2) = (0,a_1^p).
$$
Since $k$ is a perfect field, every element $(a_1,a_2)\in W_2(k)$ can be
written uniquely in the form
$$
(a_1,a_2) = (a_1,0) + p\cdot (a_2^{1/p},0),
$$
where $a_2^{1/p} \in k$ denotes the $p$-th root of $a_2$.
The element $(a,0)\in W_2(k)$ is called the Teichmüller lift of $a\in k$.
From now on, whenever we regard an element $a\in k$ as an element of
$W_2(k)$, we implicitly mean its Teichmüller lift.
Thus every element of $W_2(k)$ admits a unique expression of
the form
$$
a + p\cdot b,
$$
with $a,b\in k$ viewed inside $W_2(k)$ via their Teichmüller lifts.

The ring $W_2(k)$ has characteristic $p^2$ and comes equipped with a
surjective ring homomorphism onto $k$, sending $(a_1,a_2)$ to $a_1$.
In particular, $W_2(k)/pW_2(k)$ is isomorphic to $k$.

\subsubsection{The Weyl algebra over $W_2(k)$}

The Weyl algebra $\An(W_2(k))$ over $W_2(k)$ is defined similarly to
$\An(k)$, except that the coefficients are now taken in $W_2(k)$. 
In particular, $\An(W_2(k))$ is a free $W_2(k)$-module with basis
$$
{\bf z}^{\bf i} = z_1^{i_1} z_2^{i_2} \cdots z_{2n}^{i_{2n}}, \quad i_1, \dots, i_{2n} \in \mathbb{Z}_{\ge 0}.
$$
Any element
$$
f = \sum_{\bf i} c_{\bf i} \, {\bf z}^{\bf i} \in \An(k), \quad c_{\bf i} \in k,
$$
has an associated element
$$
[f] = \sum_{\bf i} [c_{\bf i}] \, {\bf z}^{\bf i} \in \An(W_2(k)),
$$
where $[c_{\bf i}]$ denotes the Teichmüller lift of $c_{\bf i}$. 
We say that $[f]$ is the Teichmüller lift of $f$, and we will often simply write $f$ to denote this lift. In particular, every element of $\An(W_2(k))$ admits a unique decomposition
$$
f_1 + p \cdot f_2,
$$
where $f_1, f_2 \in \An(k)$ are viewed as elements of $\An(W_2(k))$ via their Teichmüller lifts.

% \begin{lemma}
% \label{centralelement}
% Let $u$ and $v$ be elements in the center $Z$ of $\An(k)$. Then
% $[u,v]$ and $p \cdot u$ are central elements in $\An(W_2(k))$.
% \end{lemma}
% \begin{proof}
% It suffices to show that
% $[u,v]$ and $p \cdot u$ commute with every element
% $w \in \An(k)$.
% Both cases follow by observing that if $u$ is a central element in
% $\An(k)$, then
% $$
% \ad(u)(w) = [u,w] \in p \cdot \An(k)
% $$
% when viewed in $\An(W_2(k))$.
% As $\An(W_2(k))$ has characteristic $p^2$, this implies
% $$
% [p\cdot u,w] = p[u,w] = 0,
% $$
% and hence $p\cdot u$ is central.
% Moreover, since
% $$
% \ad([u,v]) = [\ad(u), \ad(v)],
% $$
% and both $\ad(u)$ and $\ad(v)$ take values in $p\cdot \An(k)$,
% we conclude that $\ad([u,v]) = 0$. Hence $[u,v]$ is central.
% \end{proof}

 \begin{lemma}
\label{centralelement}
Let $u$ and $v$ be elements of $\An(W_2(k))$ whose images
in $\An(k)$ under the natural morphism lie in $Z$.
Then $[u,v]$ is a central element of $\An(W_2(k))$.
\end{lemma}
\begin{proof}
By assumption, the operators $\ad(u)$ and $\ad(v)$
on $\An(W_2(k))$ take values in $p \cdot \An(k)$.
Since $\An(W_2(k))$ has characteristic $p^2$, we
therefore conclude that the operator
$$
\ad([u,v]) = [\ad(u), \ad(v)]
$$
is zero. Hence $[u,v]$ is central.
\end{proof}

\section{Lift of endomorphisms}
\label{sectlift}

In this section, we study when a $k$-algebra endomorphism 
$\varphi$ of $\An(k)$ admits a lift to a $W_2(k)$-algebra endomorphism 
$\Phi$ of $\An(W_2(k))$.  

Fix $\varphi$ and let
$$
u_i := \varphi(z_i), \quad i=1,\dots,2n.
$$
Viewing $u_1,\dots,u_{2n}$ as elements of $\An(W_2(k))$
via the Teichmüller lift, we have
\begin{equation}\label{commuiuj}
[u_i,u_j] = \omega_{i,j} + p \cdot u_{ij}, \quad i,j=1,\dots,2n,
\end{equation}
for some uniquely determined elements $u_{ij} \in \An(k)$.

\begin{lemma}
\label{potentiallift}
Let $v_1,\dots,v_{2n}$ be elements of $\An(k)$. Then
$$
\Phi(z_i) = u_i + p \cdot v_i, \quad i=1,\dots,2n,
$$
defines a lift of $\varphi$ if and only if
$$
u_{ij} + [u_i, v_j] - [u_j, v_i] = 0, \quad i,j=1,\dots,2n.
$$
\end{lemma}

\begin{proof}
The map $\Phi$ defines a lift of $\varphi$ if and only if
$$
[u_i + p \cdot v_i, u_j + p \cdot v_j] = \omega_{i,j}, \quad i,j=1,\dots,2n.
$$
Expanding the left-hand side gives
$$
[u_i + p \cdot v_i, u_j + p \cdot v_j] 
= \omega_{i,j} + p \cdot \big(u_{ij} + [u_i, v_j] - [u_j, v_i]\big),
$$
which immediately yields the claimed equivalence.
\end{proof}

 \begin{lemma}
\label{ppoweridcomm}
Let $u$ and $v$ be elements in $\An(W_2(k))$. Suppose that
$$
[u,v] = \kappa + p \, w,
$$
for some $w \in \An(W_2(k))$ and $\kappa \in W_2(k)$. Then
$$
[u^t, v] = t \, \kappa \, u^{\,t-1} + p \, \sum_{i=0}^{t-1} u^i w u^{\,t-i-1}, 
\quad t \in \mathbb{N}.
$$
In particular,
$$
[u^p, v] = p \, \Big( \kappa \, u^{\,p-1} + \ad(u)^{p-1}(w) \Big).
$$
\end{lemma}
\begin{proof}
We prove the first identity by induction on $t \ge 1$. The case $t=1$ is trivial. 
Assume the identity holds for some $t \ge 1$. Then
\begin{align*}
[u^{\,t+1}, v] 
&= u^t \, [u,v] + [u^t, v] \, u \\
&= u^t \, (\kappa + p \, w) + t \, \kappa \, u^t + p \sum_{i=0}^{t-1} u^i w u^{\,t-i} \\
&= (t+1) \, \kappa \, u^t + p \sum_{i=0}^{t} u^i w u^{\,t-i},
\end{align*}
as claimed.

For the second identity, we observe that the first identity implies
$$
[u^p, v] = p \, \Big( \kappa \, u^{\,p-1} + \sum_{i=0}^{p-1} u^i w u^{\,p-i-1} \Big).
$$
It therefore suffices to prove the (well known) identity
$$
\ad(\hat u)^{\,p-1}(\hat w) = \sum_{i=0}^{p-1} \hat u^i \hat w \hat u^{\,p-i-1},
$$
for $\hat u, \hat w \in \An(k)$. This follows by noting that $\ad(\hat u)$ is the difference of 
two commuting operators, multiplication by $\pm \hat u$ from the left and right, and combining this 
with the binomial identity
$ 
\binom{p-1}{i} = (-1)^i \quad \text{in } k.
$ 
\end{proof}

 \begin{proposition}
\label{propppowercomm}
Let $u$ and $v$ be elements of $\An(W_2(k))$. Suppose that
$$
[u,v] = \kappa + p \, w
$$
for some $w \in \An(W_2(k))$ and $\kappa \in W_2(k)$. Then
$$
[u^p, v^p] = p \cdot \Big( - \kappa^p + \ad(v)^{\,p-1} \ad(u)^{\,p-1}(w) \Big).
$$
\end{proposition}

\begin{proof}
By Lemma \ref{ppoweridcomm}, we have
$$
[u^p, v] = p \, \Big( \kappa \, u^{\,p-1} + \ad(u)^{\,p-1}(w) \Big).
$$
In particular, applying the same lemma to the pair $v$ and $u^p$, we obtain
\begin{align*}
[v^p, u^p] &= p \cdot \Big( - \kappa \, \ad(v)^{\,p-1}(u^{\,p-1}) - \ad(v)^{\,p-1} \ad(u)^{\,p-1}(w) \Big) \\
&= p \cdot \Big( - (p-1)! \, \kappa^p - \ad(v)^{\,p-1} \ad(u)^{\,p-1}(w) \Big) \\
&= p \cdot \Big( \kappa^p - \ad(v)^{\,p-1} \ad(u)^{\,p-1}(w) \Big),
\end{align*}
which gives the claimed identity.
\end{proof}

\begin{theorem}
\label{detofcij}
For $i,j = 1,\dots,2n$, define
$$
c_{ij} := \ad(u_i)^{\,p-1} \ad(u_j)^{\,p-1}(u_{ij})
$$
as an element in $A_n(k)$. Then $c_{ij}$ belongs to the center of $\An(k)$, and moreover
$$
p \cdot c_{ij} = [u_i^p, u_j^p] + p \cdot \omega_{i,j}  
$$
in $\An(W_2(k))$.
\end{theorem}

\begin{proof}
By Proposition \ref{propppowercomm}, we have
$$
[u_i^p, u_j^p] = p \cdot \Big( - \omega_{i,j} + \ad(u_i)^{\,p-1} \ad(u_j)^{\,p-1}(u_{ij}) \Big),
$$
since
$$
[u_i, u_j] = \omega_{i,j} + p \cdot u_{ij}.
$$
This gives the formula for $p \cdot c_{ij}$.  

 It remains to show that $c_{ij}$ is central in $\An(k)$. It follows from the formula above that this will hold if $[u_i^p, u_j^p]$ is central in $\An(W_2(k))$, which is a direct consequence of Lemma~\ref{centralelement}.
\end{proof}

%  \subsection{Lie algebra cohomology}

% Consider the abelian Lie algebra $\mathfrak{a}$ over $k$ with basis
% $$
% e_1, e_2, \dots, e_{2n}.
% $$
% Then $\An(k)$ is a module over $\mathfrak{a}$ via the action
% $$
% e_i \cdot w := [u_i, w], \quad w \in \An(k), \; i=1,\dots,2n.
% $$
% The Chevalley-Eilenberg complex for $A_n(k)$ over $\mathfrak{a}$ consists of $q$-cochains defined by
% $$
% C^{q}(\mathfrak{a},A_n(k)):=\Hom_{k}\!\bigl(\bigwedge\nolimits^{q} \mathfrak{a},\, A_n(k)\bigr)
% $$
% with coboundary operators $\delta: C^{q}(\mathfrak{a},A_n(k))\rightarrow 
% C^{q+1}(\mathfrak{a},A_n(k))$ given by
% $$
% (\delta f)(e_{i_0}\wedge \cdots\wedge e_{i_q}) = \sum_{j=0}^{q} (-1)^j \, e_{i_j} \cdot f(e_{i_0}\wedge \cdots \wedge \widehat{e_{i_j}}\wedge \cdots \wedge e_{i_q})
% $$
% for $1\leq i_1 < \cdots < i_q \leq 2n$. In this setting, $\alpha\in C^2(\mathfrak{a}, A_n(k))$ given by
% $$
% \alpha(e_i \wedge e_j) = u_{ij}
% $$
% for $1\leq i < j \leq 2n$ is a $2$-cocycle (this follows from 
% the Jacobi identity using \eqref{commuiuj}).
% The conditions in Lemma \ref{potentiallift} now says that 
% a lift of $\varphi$ exists if and only if the cohomology class of 
% $\alpha$ in $H^2(\mathfrak{a}, \An(k))$ is trivial. 
% The complex $(C^\bullet(\mathfrak{a}, A_n(k)),  \delta)$ can be 
% identified with the de Rham complex of a polynomial ring in 
% $2n$ variables.

\subsection{Obstruction interpreted via cohomology}

Lemma~\ref{potentiallift} indicates that the elements $u_{ij}$ define a cohomology class 
which encodes the obstruction to lifting $\varphi$ to an endomorphism of $\An(W_2(k))$. 
To clarify this idea, we consider the Chevalley--Eilenberg complex of $A_n(k)$, viewed 
as a module over an abelian Lie algebra of dimension $2n$, where the basis elements act 
via the operators $\ad(u_i)$. 
In an appropriate setting, the operators $\ad(u_i)$ can be interpreted as formal derivations 
of $A_n(k)$. Under this identification, the Chevalley--Eilenberg complex associated with 
the Lie algebra generated by the $\ad(u_i)$ reduces to the de Rham complex of a
polynomial ring.
 More precisely, 
recall (as explained in Section \ref{notation}) that 
$\An(k)$ is a free module over its center
$k[z_1^p,\dots,z_{2n}^p]$ with basis
$$
  u_1^{i_1} u_2^{i_2} \cdots u_{2n}^{i_{2n}},
\quad 0 \le i_1,\dots,i_{2n} < p.
$$
Equivalently, the following set of elements also forms a basis:
$$
\hat u_1^{i_1} \hat u_2^{i_2} \cdots \hat u_{2n}^{i_{2n}},
\quad 0 \le i_1,\dots,i_{2n} < p,
$$
where 
$$
\hat u_i = -u_{n+i}, \quad \hat u_{n+i} = u_i \quad \text{for } i=1,\dots,n.
$$
Moreover, since $[u_i, \hat u_j] = \delta_{i,j}$, we have
$$
\ad(u_l) (\hat u_1^{i_1} \hat u_2^{i_2} \cdots \hat u_{2n}^{i_{2n}}) = 
i_l \, \hat u_1^{i_1} \dots \hat u_l^{i_l-1} \cdots \hat u_{2n}^{i_{2n}}.
$$ 
Hence, $\ad(u_l)$ acts as the formal derivation with respect to $\hat u_l$. 

This naturally leads us to consider the $k$-linear bijection
$$
\psi \colon \An(k) \longrightarrow S = k[y_1,\dots,y_{2n}],
$$
defined by
$$
\psi\!\left(
f(z_1^p,\dots,z_{2n}^p)\,
\hat u_1^{i_1} \hat u_2^{i_2} \cdots \hat u_{2n}^{i_{2n}}
\right)
=
f(y_1^p,\dots,y_{2n}^p)\,
y_1^{i_1} y_2^{i_2} \cdots y_{2n}^{i_{2n}},
$$
where, under this correspondence, the formal derivation 
$\pdv{}{y_i}$ on the right-hand side matches the operator 
$\ad(u_i)$ on the left-hand side. 

Consider now the de Rham complex of the polynomial ring
$S = k[y_1,\dots,y_{2n}]$, which is the differential complex
$$
0 \longrightarrow S
\xrightarrow{\,d\,} \Omega^1_{S/k}
\xrightarrow{\,d\,} \Omega^2_{S/k}
\xrightarrow{\,d\,} \cdots
\xrightarrow{\,d\,} \Omega^{2n}_{S/k}
\longrightarrow 0,
$$
where 
$$
\Omega^1_{S/k} = \bigoplus_{i=1}^{2n} S \, dy_i
$$
denotes the module of Kähler differentials, and
$$
\Omega^l_{S/k} = \wedge^l_S \Omega^1_{S/k}, \quad l=1,\dots,2n.
$$
The differential $d$ is $k$-linear and is given by
$$
d\!\left(f \, dy_{i_1} \wedge \cdots \wedge dy_{i_l}\right)
=
\sum_{i=1}^{2n} \pdv{f}{y_i}   \, dy_i
\wedge dy_{i_1} \wedge \cdots \wedge dy_{i_l}.
$$

The cohomology of this complex is called the de Rham cohomology of $S$.
We denote its degree-$l$ part by $H^l_{\mathrm{dR}}(S/k)$.
Since the field $k$ has characteristic $p>0$, the de Rham cohomology of $S$
is non-trivial. More precisely, for each $l = 0,\dots,2n$, one has 
(see, e.g., \cite[Thm.~7.2]{Katz1970Nilpotent})
$$
H^l_{\mathrm{dR}}(S/k)
\cong
\bigoplus_{1 \le i_1 < \cdots < i_l \le 2n}
k[y_1^p,\dots,y_{2n}^p]\,
y_{i_1}^{p-1}\cdots y_{i_l}^{p-1}
\, dy_{i_1}\wedge\cdots\wedge dy_{i_l}.
$$
As a special case, the second de Rham cohomology group is
$$
H^2_{\mathrm{dR}}(S/k)
\cong
\bigoplus_{1 \le i < j \le 2n}
k[y_1^p,\dots,y_{2n}^p]\,
y_i^{p-1} y_j^{p-1}
\, dy_i \wedge dy_j.
$$
In particular, $H^2_{\mathrm{dR}}(S/k)$ is a free module over the subring
$k[y_1^p,\dots,y_{2n}^p]$, generated by the classes of the differential
forms
$$
y_i^{p-1} y_j^{p-1} \, dy_i \wedge dy_j,
\qquad 1 \le i < j \le 2n.
$$

\bigskip

\begin{lemma}
\label{closed2form}
The $2$-form
$$
\sum_{i < j} \psi(u_{ij}) \, dy_i \wedge dy_j \in \Omega^2(S/k)
$$
is closed.
\end{lemma}

\begin{proof}
The closedness of this $2$-form is equivalent to the identities
$$
\pdv{\psi(u_{ij})}{y_l}
+ \pdv{\psi(u_{jl})}{y_i}
- \pdv{\psi(u_{il})}{y_j} = 0,
$$
for all integers $i < j < l$.

Since $\ad(u_i)$ corresponds to the operator $\pdv{}{y_i}$ 
under the bijection $\psi$, this condition may equivalently be rewritten as
$$
\ad(u_l)(u_{ij}) + \ad(u_i)(u_{jl}) - \ad(u_j)(u_{il}) = 0,
$$
which  follows from the Jacobi identity
$$
[u_l,[u_i,u_j]] + [u_i,[u_j,u_l]] - [u_j,[u_i,u_l]] = 0
$$
and the definition $u_{st} = [u_s,u_t]$, when the calculation is performed in $\An(W_2(k))$.
\end{proof}

\begin{theorem}
\label{obstruction}
There exist elements $v_1,\dots,v_{2n}$ in $A_n(k)$ such that 
$$
u_{ij} + [u_i, v_j] - [u_j, v_i] = c_{ij} \, \hat u_i^{\,p-1} \hat u_j^{\,p-1},
$$
where $c_{ij}$ is the central element in $A_n(k)$ from Theorem \ref{detofcij}, defined by
$$
p \cdot c_{ij} = [u_i^p, u_j^p] + p \cdot \omega_{ij}.
$$
Moreover, the endomorphism $\varphi$ admits a lift to a $W_2(k)$-algebra 
endomorphism of $\An(W_2(k))$ if and only if 
$$
c_{ij} = 0, \qquad i,j=1,\dots,2n.
$$
\end{theorem}

\begin{proof}
By Lemma \ref{closed2form} and the description of $H^2_{\mathrm{dR}}(S/k)$, 
there exist polynomials $h_1,\dots,h_{2n} \in S$ and $h_{ij} \in k[y_1^p,\dots,y_{2n}^p]$, 
for $i,j=1,\dots,2n$, such that
$$
\sum_{i < j} \psi(u_{ij}) \, dy_i \wedge dy_j
=
d\Big(\sum_i h_i \, dy_i\Big) + \sum_{i < j} h_{ij} \, y_i^{p-1} y_j^{p-1} \, dy_i \wedge dy_j.
$$
Hence, for $i<j$, we have
$$
\psi(u_{ij}) = \pdv{h_j}{y_i} - \pdv{h_i}{y_j} 
+ h_{ij} \, y_i^{p-1} y_j^{p-1}, \quad i,j=1,\dots,2n.
$$
Now choose elements $v_i \in \An(k)$ and central elements $\hat c_{ij} \in Z$ such that 
$$
\psi(v_i)=-h_i, \quad \psi(\hat c_{ij})=h_{ij}, \quad i,j=1,\dots,2n.
$$
Since $\ad(u_i)$ corresponds to $\pdv{}{y_i}$ under $\psi$, we conclude
$$
u_{ij} + [u_i, v_j] - [u_j, v_i] = \hat c_{ij} \, \hat u_i^{\,p-1} \hat u_j^{\,p-1}, 
\quad i,j=1,\dots,2n,
$$
when $i<j$.
Applying the operator 
$
\ad(u_i)^{\,p-1} \, \ad(u_j)^{\,p-1}
$
to both sides, and using Theorem \ref{detofcij}, together with the facts that 
$\ad(u_i)$ and $\ad(u_j)$ commute, and that 
$\ad(u_i)^p = \ad(u_j)^p = 0$,
we obtain
$$
c_{ij} = \hat c_{ij}, \quad i,j=1,\dots,2n,
$$
for $i<j$. This proves the first statement.

It follows that if $c_{ij}=0$ for all $i,j$, then
$$
u_{ij} + [u_i, v_j] - [u_j, v_i] = 0, \quad i,j=1,\dots,2n,
$$
which by Lemma \ref{potentiallift} implies that $\varphi$ admits a lift. 
Conversely, if $\varphi$ admits a lift, then by Lemma \ref{potentiallift} 
there exist elements $w_i$, $i=1,\dots,2n$, such that
$$
u_{ij} + [u_i, w_j] - [u_j, w_i] = 0, \quad i,j=1,\dots,2n.
$$
Applying $\ad(u_i)^{p-1} \ad(u_j)^{p-1}$ and using Theorem \ref{detofcij} again, we conclude
$$
c_{ij} = 0, \quad i,j=1,\dots,2n.
$$
\end{proof}

\subsection{Lift and Poisson morpisms}
\label{sectPoisson}
  
We consider the center $Z = k[x_1,\dots,x_{2n}]$ of the Weyl algebra $A_n(k)$ as a Poisson algebra with the standard bracket
$$
\{ f , g \} = \sum_{l=1}^n 
\left(
\pdv{f}{x_l} \pdv{g}{x_{l+n}}
-
\pdv{f}{x_{l+n}} \pdv{g}{x_l}
\right),
$$
or equivalently,
$$
\{f,g\} = (\nabla f)^{T} \, \omega^{-1} \, (\nabla g),
$$
where the gradient of $h \in Z$ is defined by
$$
\nabla h =
\begin{pmatrix}
\pdv{h}{x_1} \\[1mm]
\vdots \\[1mm]
\pdv{h}{x_{2n}}
\end{pmatrix},
$$ 
and $\omega$ is the matrix
$$
\omega =
\begin{pmatrix}
0 & -I_n \\[1mm]
I_n & 0
\end{pmatrix} \in \Mat_{2n}(k).
$$
A $k$-algebra endomorphism $\phi: Z \to Z$ is called a  Poisson morphism  if it preserves the bracket:
$$
\phi(\{f,g\}) = \{\phi(f),\phi(g)\}, \quad \text{for all } f,g \in Z.
$$
Equivalently, this condition can be expressed in terms of the Jacobian matrix 
$$
J_\phi = \Bigl(\pdv{\phi(x_i)}{x_j}\Bigr)
$$
as
\begin{equation}
\label{poissonmorp}
J_\phi \, \omega^{-1} \, J_\phi^T = \omega^{-1}.
\end{equation}
Since $\omega$ is invertible, we immediately see that $\det J_\phi = \pm 1$, 
which implies that a Poisson morphism is étale. Furthermore, (\ref{poissonmorp}) is equivalent to the identity
\begin{equation}
\label{2-formstable}
J_\phi^T \, \omega \, J_\phi = \omega,
\end{equation}
which expresses the invariance of the standard $2$-form
$$
\sum_{l=1}^n dx_l \wedge dx_{l+n}
$$
under $\phi$, that is,
$$
\sum_{l=1}^n dx_l \wedge dx_{l+n} = \sum_{l=1}^n d\phi(x_l) \wedge d\phi(x_{l+n}).
$$
 
One may relate the above Poisson bracket on $Z$ to commutators in 
$\An(W_2(k))$. More precisely, if $f$ and $g$ are elements in
$Z$ then the commutator
$[f,g]$
in $\An(W_2(k))$ is contained in $p \cdot A_n(W_2(k))$. 
Moreover $[f,g]$ is a central element in $A_n(W_2(k))$
by Lemma \ref{centralelement}, so there exists a 
unique $h \in Z$ 
such that 
$$[f,g] = p \cdot h.$$
Actually 
$ 
h =  \{ f, g \}
$ 
or in other words
\begin{align}
\label{idpoiscomm}
\{ f, g \} =   \frac{[f,g]}{p} \in Z, 
\end{align}
where the right hand side should be interpreted as explained 
above. Actually the right hand side of (\ref{idpoiscomm}) is 
obviously a derivation in both entries, so for the identity 
in (\ref{idpoiscomm}) to be satisfied, it suffices that 
it is satisfied, when $f$ and $g$ are both one of the 
coordinate functions $x_1,\dots,x_{2n}$. In this case,
the identity is a consequence of Theorem \ref{obstruction}
applied to the case, where $\varphi$ is the identity 
morphism. 

\bigskip 

We now relate the remarks above to the morphism $\varphi$.
Let
$$
J_\varphi = \left( \pdv{\varphi(x_i)}{x_j} \right) \in \Mat_{2n}(Z)
$$
denote the Jacobian matrix associated with the restriction
$\varphi_Z$ of $\varphi$ to the center $Z$, and let
$$
C = (c_{ij}) \in \Mat_{2n}(Z)
$$
be the matrix whose $(i,j)$-entry is the element $c_{ij} \in Z$
defined in Theorem \ref{detofcij}.

\begin{theorem}
\label{theoremidjac}
One has
$$
J_\varphi \, \omega^{-1} \, J_\varphi^T = \omega^{-1} + C.
$$
In particular, the induced morphism $\varphi_Z$ is a Poisson morphism
if and only if $\varphi$ admits a lift to a $W_2(k)$-algebra
endomorphism of $\An(W_2(k))$.
\end{theorem}
\begin{proof}
The $(i,j)$-th entry of 
$
J_\varphi \, \omega^{-1} \, J_\varphi^T
$
equals 
$
\{\varphi(x_i), \varphi(x_j)\},
$
which, by identity~\eqref{idpoiscomm}, equals
$$
\frac{[\varphi(x_i), \varphi(x_j)]}{p}.
$$
By Theorem~\ref{detofcij}, the latter equals
$
-\omega_{i,j} + c_{ij},
$
which coincides with the $(i,j)$-th entry of
$
\omega^{-1} + C.
$
This proves the first part of the claimed statement. In particular, by the discussion above, $\varphi_Z$ is a Poisson morphism if and only if $C$ is the zero matrix. Moreover, by Theorem~\ref{obstruction}, $C=0$ if and only if $\varphi$ admits a lift to a $W_2(k)$-algebra endomorphism of $\An(W_2(k))$.
\end{proof}

\section{Tsuchimoto's approach and its refinements}

For completeness, we include here a description of Tsuchimoto's approach
in \cite{Tsuchimoto2005Endomorphisms}. 
Most of this is contained in 
Tsuchimoto's paper, but the setup there, 
involving connections on vector bundles, is slightly different 
from the one presented below.

\subsection{A trivialization of  $\An(k)$}

The Weyl algebra $\An(k)$ is an Azumaya algebra and as such 
it has a trivialization. We will work with a specific 
trivialization which will be described in the following.

Recall that the center $Z$ of $\An(k)$, is a 
polynomial ring over $k$ in $2n$ variables, generated by 
$$ x_i := z_i^p, \qquad i=1,\dots,2n.$$
The polynomial ring $S = k[y_1, y_2, \dots, y_{2n}]$ 
in $2n$ variables will in the following be viewed 
as a faithfully flat
extension of $Z$ via the identifications
$$ y_i^p = z_i^p, \qquad i=1,\dots,2n.$$
Then $\An(k) \otimes_Z S$ is 
isomorphic to a matrix algebra $\Mat_{p^n}(S)$ over
$S$. More precisely, let $M$ denote the 
$S$-module
$$ M = S[T_1,T_2,\dots,T_n]/(T_1^p, T_2^p, \dots,T_n^p).$$
Then $M$ comes with two types of natural $S$-linear 
endomorphisms; multiplication by $T_i$, denoted
by $\nu_i$, and formal derivative with respect to $T_i$,
denoted by $\nu_{i+n}$.  Then 
\begin{align*}
  \An(k) \otimes_ Z S & \rightarrow \End_S(M), \\
  z_i & \mapsto \nu_i + y_i, 
\end{align*}
defines an explicit isomorphism of $S$-algebras
\cite[Lemma 5]{Tsuchimoto2003Preliminaries}. From now on, we  identify $\An(k)$ with a subset 
 of $\End_S(M)$ by this isomorphism.

\subsection{$\An(k)$-linear endomorphisms of $M$}

The $k$-algebra $\End_Z(M)$ of $Z$-linear endomorphisms of $M$ plays an important role 
in describing the endomorphisms of $\An(k)$. In this section, we study certain special 
elements of $\End_Z(M)$. Recall that $\An(k)$ is viewed as a subalgebra of $\End_S(M)$, 
which in turn is a subalgebra of $\End_Z(M)$.

Let $\pdv{}{y_i}$
 denote formal partial derivatives with respect to
$y_i$ regarded as an element in $\End_Z(M)$.
Define $Z$-linear 
endomorphisms
$$
     L_i  =  \pdv{}{y_i} + \hat{\nu}_i,
$$
for $i=1,2,\dots,2n$, where 
$$ \hat{\nu}_i = \begin{cases}
 - \nu_{i+n}, & \text{if $i=1,\dots,n$,} \\
 \nu_{i-n}, & \text{if $i=n+1,\dots,2n$.}
\end{cases}$$
Note that 
$$ [\hat \nu_i, \nu_j] = -\delta_{i,j}.$$
In particular,
$$[L_i, z_j] = [\pdv{}{y_i} + \hat{\nu}_i, \nu_i+y_i]=0, $$
for all $i,j =1,\dots,2n$, and 
hence each $L_i$ is an element $\End_{\An(k)}(M)$.

\begin{remark}
In fact, the elements $L_1,\dots,L_{2n}$, together with $S$, generate
$\End_{\An(k)}(M)$. Moreover, the centralizer of
$\End_{\An(k)}(M)$ in $\End_Z(M)$ equals $\An(k)$. Since this result will not be
needed here, we omit the proof.
\end{remark}

\subsection{Endomorphisms of $\An(k)$}

Let $\varphi$ denote a $k$-algebra endomorphism of $\An(k)$. 
As explained in Section~\ref{notation}, the center $Z$ is preserved by $\varphi$. 
Thus there is a natural extension of $\varphi_Z$ to a $k$-algebra endomorphism of $S$. 
Combining this with the identification
$$
\End_S(M) \simeq \An(k) \otimes_Z S,
$$
we obtain an induced $k$-algebra endomorphism of $\End_S(M)$, 
which, by abuse of notation, we also denote by $\varphi$.

Once we fix an $S$-basis of $M$, we may identify $\End_S(M)$
with the matrix ring $\Mat_{p^n}(S)$ and $\varphi$ with an
endomorphism of $\Mat_{p^n}(S)$.  Here we fix the monomial
basis 
$$
T_1^{i_1} \cdots T_n^{i_n}, \qquad 0 \leq i_1,\dots,i_n < p,
$$
of $M$. In this basis $\nu_i$, for $i=1,\dots,2n$, have matrix
representations with coefficients in $\ZZ/p\ZZ$.  Therefore, 
by the lemma below, there exist an invertible element 
$G \in \End_S(M)$ such that
$$
\varphi(z_i) = G \cdot \nu_i \cdot G^{-1} + \overline{y_i},
\qquad i=1,\dots,2n,$$
where we use the notation $\overline{r} : = \varphi(r) 
\in S$ whenever $r \in S$.

\medskip
 
 This following lemma
should be  well-known; however, since we were unable to find 
a precise reference, we provide a proof below.

\begin{lemma}
Let $R$ be a unique factorization domain, and
let $\phi: \Mat_m(R) \rightarrow \Mat_m(R)$ be a ring homomorphism. Then there exists 
an invertible matrix $G \in \Mat_m(R)$ and a ring endomorphism 
$\phi_c: R \rightarrow R$ such that 
$$ \phi \big ( (a_{ij}) \big) = G \cdot \big( \phi_c(a_{ij}) \big) \cdot G^{-1},$$
for all $(a_{ij}) \in \Mat_m(R)$.  
Moreover, the matrix $G$ is 
unique up to multiplication by a  units in $R$; equivalently its image in $\PGL_m(R)$ 
is uniquely determined. 
\end{lemma}
\begin{proof}
Let $E_{ij}$, for $1 \leq i, j \leq m$, denote the $(i,j)$-th standard matrix
unit in $\Mat_m(R)$ and let $F_{ij} = \phi(E_{ij})$. Then $\sum_i F_{ii} = 1$,
so in particular $F_{ii} \neq 0$ for some $i$. Since $F_{ii} = F_{i1} \, F_{11} \, F_{1i}$, we must also have $F_{11} \neq 0$. Fix 
a nonzero ${\bf r}=(r_1,\dots, r_m) \in F_{11} R^m$ such that the
entries $r_1, r_2, \dots, r_m
\in R$ have no nontrivial common divisors. This is possible,
since if ${\bf r} = \beta \, {\bf r}'$, for ${\bf r}' \in R^m$ and 
$\beta \in R$, then ${\bf r}' = F_{11} {\bf r}'$.  For $i=1,\dots,m,$ set ${\bf v}_i = 
F_{i1} {\bf r}$. Then we have the relations 
$$ F_{ij} \, {\bf v}_l = \delta_{jl} \cdot {\bf v}_i.$$

We claim that the vector ${\bf v}_1, {\bf v}_2, \dots, {\bf v}_m$ 
form a basis for the free $R$-module $R^m$. First,  linear 
independence follows as
$$ F_{1i} \, (\alpha_1 \cdot {\bf v}_1 + \alpha_2 \cdot {\bf v}_2 + \cdots + \alpha_m \cdot {\bf v}_m) = \alpha_i \cdot {\bf r},$$
valid for all $\alpha_1, \dots, \alpha_m \in R$. Next, let ${\bf v} \in R^m$. Since  
the $m+1$ elements ${\bf v}_1, {\bf v}_2, \dots, {\bf v}_m, {\bf v}$ are linearly dependent over $R$,  there exists a nonzero element $c \in R$ such that 
$$ c \cdot {\bf v} =  \beta_1 \cdot {\bf v}_1 + \beta_2 \cdot {\bf v}_2 + \cdots + \beta_m \cdot {\bf v}_m,$$
for some $\beta_1, \beta_2, \dots, \beta_m \in R$. Applying $F_{1i}$
to both sides yields $c \cdot F_{1i} {\bf v} = \beta_i \cdot 
{\bf r}$. By the choice of ${\bf r}$, it follows  that 
$c$ divides each $\beta_i$.
Hence ${\bf v}$ lies in the $R$-span of ${\bf v}_1, {\bf v}_2, \dots, {\bf v}_m$, proving 
that these vectors form a basis of $R^m$. 

Let $G \in \Mat_m(R)$ denote the matrix whose columns are ${\bf v}_1, 
{\bf v}_2, \dots, {\bf v}_m$. Then $F_{ij}=G \cdot  E_{ij} \cdot G^{-1}$ 
as both sides of the equation agree  on the basis ${\bf v}_1, {\bf v}_2, \dots, {\bf v}_m$. Now recall that the centralizer of the set $\{E_{ij} \}$ equals
the center of $\Mat_m(R)$. Since 
conjugation preserves centralizers, the centralizer of the set $\{ F_{ij}
\}$
is therefore also   equal to the center of $\Mat_m(R)$. Consequently, the map
$\phi$ must map the center of $\Mat_m(R)$ to itself; that is $\phi( \alpha 
\cdot {\rm I} )$, for $\alpha \in R$, must be of the form $\beta \cdot
{\rm I}$, for some $\beta \in R$. This allows us to  define
$\phi_c$ by $\phi_c(\alpha) = \beta$. With his notation, we 
obtain
$$ \phi \big ( (a_{ij}) \big) = G \cdot \big( \phi_c(a_{ij}) \big) \cdot G^{-1},$$
as desired. In particular, $G$ must satisfy  
$$ \phi(E_{ij}) = G  \cdot E_{ij} \cdot G^{-1},$$
which implies that $G$ is determined up to multiplication by an unit in the centralizer
of the set $\{E_{ij} \}$; that is up to a unit in $R$. 
\end{proof}

\begin{remark}
\label{remakrtrace}
Let $n=1$ and consider the trace in $\End_S(M)$ of an element of the form
$$
\nu_1^{i_1} \nu_2^{i_2}, \quad 0 \leq i_1, i_2 < p.
$$
It is a simple exercise to see that
$$
\Tr(\nu_1^{i_1} \nu_2^{i_2}) = -\delta_{i_1,p-1} \, \delta_{i_2,p-1}.
$$
Since trace is invariant under conjugation by $G$, we therefore also have
$$
\Tr\big(\varphi(z_1)^{i_1} \varphi(z_2)^{i_2}\big) =
\Tr\big((\nu_1 + \overline{y_1})^{i_1} (\nu_2 + \overline{y_2})^{i_2}\big)
= -\delta_{i_1,p-1} \, \delta_{i_2,p-1}.
$$
For general $n$, we may consider $M$ as an $n$-fold tensor product of the case $n=1$
above. Therefore
$$
\Tr\big(\varphi(z_1)^{i_1} \cdots \varphi(z_{2n})^{i_{2n}}\big)
= (-1)^n \, \delta_{i_1,p-1} \cdots \delta_{i_{2n},p-1},
\quad 0 \leq i_1, \dots, i_{2n} < p.
$$
As mentioned in Section~\ref{notation}, elements of the form
$$
\varphi(z_1)^{i_1} \cdots \varphi(z_{2n})^{i_{2n}}, \quad 0 \leq i_1, \dots, i_{2n} < p,
$$
constitute a basis for $\An(k)$ as a module over $Z$.  
When writing an element $u \in \An(k)$ in this basis, it follows that $(-1)^n \Tr(u)$ is the coefficient of the basis element
$$
\varphi(z_1)^{p-1} \cdots \varphi(z_{2n})^{p-1}.
$$
Hence,
$$
\left( \ad(\varphi(z_1))^{p-1} \cdots \ad(\varphi(z_{2n}))^{p-1} \right)(u) = (-1)^n \Tr(u),
$$
so that the operator
$$
\ad(\varphi(z_1))^{p-1} \cdots \ad(\varphi(z_{2n}))^{p-1}
$$
does not depend on $\varphi$. In particular,
$$
\ad(\varphi(z_1))^{p-1} \cdots \ad(\varphi(z_{2n}))^{p-1} = \ad(z_1)^{p-1} \cdots \ad(z_{2n})^{p-1}.
$$
\end{remark}

\subsection{Differential equations}

We fix an endomorphism $\varphi$ of $\An(k)$ and 
the associated notation as above.

 \begin{proposition}
\label{basisofAlinearphi}
For $i=1,\dots,2n$, the elements 
$$ 
    L_{\varphi,i}  = G \,  \Biggl( \pdv{}{y_i} + \sum_{l=1}^{2n} 
\pdv{\overline{y}_l}{y_i} \hat \nu_l  \Biggr) \, G^{-1}, 
$$ 
in $\End_Z(M)$ belong to $\End_{\An(k)}(M)$.
\end{proposition}

\begin{proof}
As noted in Section \ref{notation}, the algebra
$\An(k)$ is generated as a $Z$-algebra by
$$\varphi(z_1), \dots, \varphi(z_n), \varphi(\partial_1), \dots, \varphi(\partial_n).$$ 
Hence if an element in $\End_Z(M)$ is $\varphi(\An(k))$-linear, then it is also 
$\An(k)$-linear.
Therefore
it suffices to show that 
all $L_{\varphi,i}$  commute 
with 
$$\varphi(z_1), \dots, \varphi(z_n), \varphi(\partial_1), \dots, \varphi(\partial_n).$$

For $j=1,\dots,2n$ we have 
$$\varphi(z_j) = G \cdot ( \nu_j  + \overline{y_j} ) \cdot G^{-1},$$
and a direct computation shows
$$\left[ \nu_j + \overline{y_j}, 
\pdv{}{y_i} + \sum_{l=1}^n   \pdv{\overline{y}_l}{y_i} \hat \nu_l  \right] = 0.$$
Since conjugation by $G$ preserves commutators, it follows that
$$[\varphi(z_j), L_{\varphi,i}] = 0$$
for $j=1, \dots, 2n$, as required.  
\end{proof}

Note that $L_{\varphi,i}$ coincide with 
$L_i$, when $\varphi$ is the identity morphism.
For a general endomorphism $\varphi$, the elements
 $L_{\varphi,i}$  
remain naturally related to $L_i$, as explained 
in the following result.

\begin{corollary}
\label{gammadef}
For $i=1,\dots,2n$, we have
$$ 
[L_{\varphi, i}, y_j]  =  \delta_{i,j}.
 $$
In particular, there exist elements $\gamma_1,\dots,\gamma_{2n}
\in S$ such that
$$ 
L_{\varphi,i}  = L_i + \gamma_i,  
$$ 
for $i=1,\dots,2n$.
\end{corollary}

\begin{proof}
The commutator relations follow directly from the definition  
$L_{\varphi, i}$, since $G$ and $\hat \nu_i$ 
all commute with $S$. By the observation preceding the corollary, 
analogous relations hold for $L_i$. 
Hence, the differences $L_{\varphi,i}-L_i$ 
commute with all $y_j$, for $j=1,\dots,2n$, and thus lie in $\End_S(M)$. 
As $\End_S(M)$ is generated by $\An(k)$ and $S$, 
it follows from Proposition \ref{basisofAlinearphi} that these differences 
are elements of the center $S$ of $\End_S(M)$, which proves the existence 
of $\gamma_i$..
\end{proof}

An important insight of Tsuchimoto is that the $p$-power map on
$\End_S(M)$ determines the elements $\gamma_i$
appearing in Corollary~\ref{gammadef} by a formula that does not involve $G$.
To explain this, we first recall the following result due to N.~Jacobson
(\cite[Chapter~V, \S7]{JacobsonLieAlgebras}).

\begin{lemma} 
\label{pthpowercomm}
Let $R$ be a $\ZZ/p\ZZ$-algebra, and let 
$C \subseteq R$ be a commutative subalgebra. 
Let $a, b \in R$ with $b \in C$, and assume that 
$\ad(a)(C) \subseteq C$. Then
$$ (a+b)^p = a^p + b^p + \ad(a)^{p-1}(b).$$
 \end{lemma}

As a consequence we obtain the following result.

\begin{proposition}
\label{prop:p-power-identities}
For $i=1,\dots,n$, we have the following identities in $S$:
$$
L_{\varphi,i}^p
=
\pdv[p-1]{ }{y_i}
\Biggl(
\sum_{l=1}^n
\pdv{\overline{y}_l}{y_i}\, \overline{y}_{l+n}
\Biggr).
$$
Here, the right-hand side should be understood as the $(p-1)$-st partial derivative
of the polynomial inside the parentheses.
\end{proposition}
\begin{proof}
Set
$$
a := \pdv{}{y_i}
+ \sum_{l=n+1}^{2n}
\pdv{\overline{y}_l}{y_i} \hat \nu_l,
\qquad
b :=   \sum_{l=1}^{n}
\pdv{\overline{y}_l}{y_i} \hat \nu_l.
$$
Applying Lemma~\ref{pthpowercomm} with $C \subseteq \End_Z(M)$
being the commutative subalgebra generated by $S$ and
$\hat \nu_{1},\dots,\hat \nu_{n}$ gives
$$
(a+b)^p = a^p + \ad(a)^{p-1}(b),
$$
since $b^p=0$. Next, write
$$
a' := \pdv{}{y_i},
\qquad
b' := \sum_{l=n+1}^{2n}
\pdv{\overline{y}_l}{y_i} \hat \nu_l 
$$
and apply Lemma~\ref{pthpowercomm} again with the commutative subalgebra generated by
$S$ and $\hat \nu_{n+1},\dots,\hat \nu_{2n}$, yielding
$$
a^p
=
(a')^p + (b')^p + \ad(a')^{p-1}(b')
=
\sum_{l=n+1}^{2n}
\pdv[p]{\overline{y}_l}{y_i} \hat \nu_l
=
0,
$$
since $(a')^p=(b')^p=0$. Hence
$$
(a+b)^p = \ad(a)^{p-1}(b).
$$
Now decompose
$$
b = b_1 + b_2, \quad
b_1 :=   \sum_{l=1}^{n}
\pdv{\overline{y}_l}{y_i}
\Bigl( \hat \nu_l  - \overline{y}_{n+l} \Bigr),
\qquad
b_2 := \sum_{l=1}^n
\pdv{\overline{y}_l}{y_i}\, \overline{y}_{n+l}.
$$
Since $[\hat \nu_{n+j}, \hat \nu_l] = \delta_{j,l}$, for
$j,l=1,\dots,n$, we find
$$
\ad(a)\Bigl( \hat \nu_l - \overline{y}_{n+l} \Bigr) = 0,
$$
and thus
$$
\ad(a)^{p-1}(b_1)
=
 \sum_{l=1}^n
\ad(a)^{p-1}\Bigl( \pdv{\overline{y}_l}{y_i} \Bigr)
\Bigl( \hat \nu_l - \overline{y}_{n+l} \Bigr)
=
  \sum_{l=1}^n
\pdv[p]{\overline{y}_l}{y_i}
\Bigl(\hat \nu_l - \overline{y}_{n+l} \Bigr)
=
0.
$$
Therefore,
$$
(a+b)^p
=
\ad(a)^{p-1}(b_2)
=
\pdv[p-1]{ }{y_i}
\Biggl(
\sum_{l=1}^n
\pdv{\overline{y}_l}{y_i}\, \overline{y}_{n+l}
\Biggr).
$$
Finally, since $L_{\varphi,i} = G (a+b) G^{-1}$ and $G$ commutes with $S$, we obtain
$$
L_{\varphi,i}^p
=
G (a+b)^p G^{-1}
=
\pdv[p-1]{ }{y_i}
\Biggl(
\sum_{l=1}^n
\pdv{\overline{y}_l}{y_i}\, \overline{y}_{n+l}
\Biggr),
$$
as claimed.
\end{proof}

Lemma \ref{pthpowercomm} also implies the following result.

\begin{proposition}
\label{prop:p-power-identities-scalar}
For $i=1,\dots,2n$ and $\gamma \in S$, we have 
$$
(L_{\varphi,i}+ \gamma)^p  =  L_{\varphi,i}^p + \gamma^p + 
\pdv[p-1]{\gamma}{y_i}.$$
\end{proposition}
\begin{proof}
The identity follows by applying Corollary \ref{gammadef} and
Lemma~\ref{pthpowercomm} with $C=S$, $a=L_{\varphi,i}$, and $b=\gamma$.
\end{proof}

\begin{corollary}
\label{constants}
The elements $\gamma_1,\dots,\gamma_{2n}$ 
described in Corollary~\ref{gammadef} satisfy
the following differential equations:
$$
\pdv[p-1]{ }{y_i}
\Biggl(
\sum_{l=1}^{n}
\pdv{\overline{y}_l}{y_i}\, \overline{y}_{n+l}
\Biggr)
=
\gamma_i^p +
\pdv[p-1]{\gamma_i}{y_i},
 $$
Moreover, these differential equations uniquely determine
$\gamma_i$  as elements of $S$.
\end{corollary}
\begin{proof}
Recall that $L_i$ coincides with $L_{\varphi,i}$ when $\varphi$
is the identity morphism. Thus, combining
Proposition~\ref{prop:p-power-identities}
and Proposition~\ref{prop:p-power-identities-scalar}
(in the case where $\varphi$ is the identity morphism) yields
$$
(L_i + \gamma_i)^p
=
\gamma_i^p +
\pdv[p-1]{\gamma_i}{y_i}.
$$
But $L_{\varphi,i} = L_i + \gamma_i$,
so the claimed differential equations follow by another application
of Proposition~\ref{prop:p-power-identities}
(now for the morphism $\varphi$).

A degree argument in $S$ shows that the homogeneous equation
$$
\gamma^p + \pdv[p-1]{\gamma}{y_i} = 0
$$
only admits the trivial solution $\gamma=0$ in $S$.
Hence $\gamma_i$ is uniquely determined.
\end{proof}

\subsection{Conditions for preserving the Poisson structure}

In this subsection, we discuss bounds on the characteristic
$p$ that ensure that the induced morphism $\varphi_Z$ of the
center is a Poisson morphism.
We use the notation and remarks introduced in Section~\ref{sectPoisson}. The main goal of this section
is to restate and slightly strengthen a result of Tsuchimoto
(\cite[Cor.~3.3]{Tsuchimoto2005Endomorphisms}).

\begin{theorem}
 \label{etalecenter}
Let $\varphi$ be a $k$-algebra endomorphism of $\An(k)$
as in the preceding subsections. If
$$
\deg(\varphi(z_l)) + \deg(\varphi(z_{n+l})) < 2p
$$
for $l=1,\dots,n$, then the restriction $\varphi_Z$
of $\varphi$ to the center $Z$ is a Poisson morphism.
In particular, the restriction $\varphi_Z$ is étale and preserves 
the standard $2$-form on $Z$.
\end{theorem}
We will prove this result below. Tsuchimoto proves this result 
under the stronger assumptions
$$
\deg(\varphi(z_l)) < p/2, \quad l=1,\dots,2n.
$$
Notice that the streghthen version 
also applies in case 
$$
\deg(\varphi(z_l)) < p, \quad l=1,\dots,2n.
$$
The following example says that the streghthen version is 
optimal in some sense.

\begin{example}
Let $\varphi$ denote an endomorphism of the first Weyl 
algebra $A_1(k)$ of the form
$$ \varphi(z_1) = z_1, \qquad \varphi(z_2) = z_2 + z_2^p z^i_1,$$
for some integer $0 \leq i \leq p-1$. Then by Lemma \ref{pthpowercomm}
$$\varphi(z_1^p) = z_1^p, \qquad \varphi(z_2) = z_2^p (1- \delta_{i,p-1})+ z_2^{p^2} 
z_1^{ip}.$$
It follows that $\varphi$ is étale on the center if and only
if $i<p-1$. Note that the case $i=p-1$ is not within reach 
of  Theorem \ref{etalecenter} since  
$$\deg(\varphi(z_1)) + \deg(\varphi(z_2)) = p+1+i.
$$
Notice that when $i=0$ then 
$$
\Phi(z_1) = z_1 - p \cdot z_2^{p-1} z_1, \quad 
\Phi(z_2) = z_2 + z_2^p,
$$
defines a lifting of $\varphi$ to an endomorphism 
$\Phi$ of $\An(W_2(k))$. 
For $0 < i < p-1$, liftings also exist, although they are more complicated.
\end{example}

\bigskip 

\bigskip

Recall that by Corollary \ref{gammadef} there exist elements
$\gamma_1,\dots,\gamma_{2n} \in S$ such that
$$
L_{\varphi, i} = L_i + \gamma_i, \qquad i=1,\dots,2n.
$$
Let $J_\gamma$ denote the matrix
$$
J_\gamma = \Bigl(\pdv{\gamma_i}{y_j}\Bigr) \in \mathrm{Mat}_{2n}(S),
$$
and let $\hat J_\varphi$ denote the Jacobian matrix of the morphism induced by $\varphi$ on $S$:
$$
\hat J_\varphi = \Bigl(\pdv{\overline{y_i}}{y_j}\Bigr).
$$

\begin{corollary}
\label{matrix-id}
For $i,j=1,\dots,2n$, the following commutation relation holds:
$$
[ L_{\varphi, i}, L_{\varphi, j}] =
\sum_{l=1}^n
\left(
\pdv{\overline{y}_{n+l}}{y_i} \pdv{\overline{y}_l}{y_j}
-
\pdv{\overline{y}_l}{y_i} \pdv{\overline{y}_{n+l}}{y_j}
\right).
$$
In particular, 
$$
\hat J_\varphi^T \, \omega \, \hat J_\varphi = \omega + 
J_\gamma^T - J_\gamma.
$$
\end{corollary}
\begin{proof}
Using that
$$
[\hat \nu_s, \hat \nu_t] = \omega_{s,t}, \qquad s,t=1,\dots,2n,
$$ 
a straightforward computation shows that
$$
\left[
\pdv{}{y_i} + \sum_{l=1}^{2n} \pdv{\overline{y}_l}{y_i} \hat \nu_l,
\pdv{}{y_j} + \sum_{l=1}^{2n} \pdv{\overline{y}_l}{y_j} \hat \nu_l
\right]
=
\sum_{l=1}^n
\left(
\pdv{\overline{y}_{n+l}}{y_i} \pdv{\overline{y}_l}{y_j}
-
\pdv{\overline{y}_l}{y_i} \pdv{\overline{y}_{n+l}}{y_j}
\right).
$$
Hence,
$$
[ L_{\varphi, i}, L_{\varphi, j}] =
\sum_{l=1}^n
\left(
\pdv{\overline{y}_{n+l}}{y_i} \pdv{\overline{y}_l}{y_j}
-
\pdv{\overline{y}_l}{y_i} \pdv{\overline{y}_{n+l}}{y_j}
\right),
$$
since $G$ commutes with $S$. It follows that 
$[ L_{\varphi, i}, L_{\varphi, j}]$ is the $(i,j)$-th entry of the matrix
$\hat J_\varphi^T \, \omega \, \hat J_\varphi$.
Applying this conclusion to the case where  $\varphi$ is the 
identity morphism gives
$$
[ L_i, L_j] = \omega_{i,j}.
$$
Finally, since $L_{\varphi, i} = L_i + \gamma_i$, we have
$$
[ L_{\varphi, i}, L_{\varphi, j}] = [L_i+\gamma_i, L_j + \gamma_j] =[L_i,L_j] + [L_i, \gamma_j] - [L_j, \gamma_i] = \omega_{i,j} + \pdv{\gamma_j}{y_i} - \pdv{\gamma_i}{y_j},
$$
which yields the final identity stated in the corollary.
\end{proof}

We are now ready to prove Theorem \ref{etalecenter}.

\begin{proof}
The stated assumption is equivalent to the requirement that 
$$
\deg(\overline{y}_l) + \deg(\overline{y}_{n+l}) < 2p
$$
for each $l = 1, \dots, n$, where we use the standard degree 
on the polynomial ring $S$. It follows that the left-hand side
$$
\pdv[p-1]{ }{y_i}
\Biggl(
\sum_{l=1}^{n}
\pdv{\overline{y}_l}{y_i}\, \overline{y}_{n+l}
\Biggr)
$$
of the differential equation defining $\gamma_i$ in
Corollary \ref{constants} has degree strictly less than $p$.
On the other hand, the right-hand side
$$
\gamma_i^p + \pdv[p-1]{\gamma_i}{y_i}
$$
has degree at least $p$ unless $\deg(\gamma_i) \le 0$.
In particular, $\gamma_i$ must be a constant in $S$, and hence
the matrix $J_\gamma$ is zero.

It now follows from Corollary \ref{matrix-id} that
$$
\hat J_\varphi^T \, \omega \, \hat J_\varphi = \omega.
$$
Let $J_\varphi$ denote the Jacobian matrix associated with the restriction 
$\varphi_Z$ of $\varphi$ to the center. Then
$$
J_\varphi = \hat J_\varphi^{[p]} = \left( \left( \pdv{\overline{y}_i}{y_j} \right)^p \right),
$$
and hence
$$
J_\varphi^T \, \omega \, J_\varphi = \omega.
$$
The desired conclusion now follows from the discussion in Section \ref{sectPoisson}.
\end{proof}

\begin{remark}
In the case $n=1$, we only need $c_{12}$ to be zero in order to lift 
$\varphi$ to an endomorphism of $\An(W_2(k))$. Recall that
$$
c_{12} = \ad(u_1)^{p-1} \ad(u_2)^{p-1} (u_{12}),
$$
where $u_1 = \varphi(z_1)$ and $u_2 = \varphi(z_2)$, and the commutator 
of $u_1$ and $u_2$ in $\An(W_2(k))$ determines $u_{12} \in \An(k)$ 
by the identity
$$
[u_1, u_2] = 1 + p \cdot u_{12}.
$$
By the defining relations in $\An(W_2(k))$, we furthermore know that
$$
\deg(u_{12}) \leq \deg(u_1) + \deg(u_2) - 2.
$$
In particular, under the same assumptions as in Theorem~\ref{etalecenter}, we have
$$
\deg(u_{12}) < 2p - 2.
$$
Hence
$$
c_{12} = \ad(u_1)^{p-1} \ad(u_2)^{p-1} (u_{12})
= \ad(z_1)^{p-1} \ad(z_2)^{p-1} (u_{12}) = 0,
$$
where we have used the conclusion of Remark~\ref{remakrtrace}.
It follows from Theorem~\ref{obstruction} that $\varphi$ admits a lift to 
$\An(W_2(k))$. Therefore, when $n=1$, we do not need the differential 
equations in Theorem~\ref{constants} to conclude the statement of 
Theorem~\ref{etalecenter}.
\end{remark}

There are examples of endomorphisms $\varphi$ which are not 
Poisson morphisms on the center. Surprisingly, one may even find 
examples where $\varphi$ is an isomorphism, as the following
example from \cite{BelovKontsevich2005} illustrates.

\begin{example}
\label{exnonpoisson}
Consider the second Weyl algebra $A_2(k)$, generated by 
$z_1, z_2, z_3,$ and $z_4$, and the automorphism 
$\varphi$ of $A_2(k)$ defined by
$$
\varphi(z_1) = z_1 + z_2^p z_3^{p-1}, \quad 
\varphi(z_i) = z_i \text{ for } i=2,3,4.
$$
Then, by Lemma~\ref{pthpowercomm},
$$
\varphi(x_1) = x_1 + x_2^p x_3^{p-1} - x_2, \quad
\varphi(x_i) = x_i \text{ for } i=2,3,4.
$$
In particular,
$$
c_{14} = \{\varphi(x_1), \varphi(x_4)\} = -1,
$$
which shows that $\varphi_Z$ is not a Poisson morphism.
In this case, the only nonzero entries of the matrix $C$ from Theorem~\ref{theoremidjac} 
are $c_{41} = 1$ and $c_{14} = -1$. Moreover, one may check that 
$\gamma_3 = y_2$, while $\gamma_1 = \gamma_2 = \gamma_4 = 0$.
\end{example}

In conclusion, it is worth mentioning that the matrix $C$ introduced in
Section \ref{sectlift} (see Theorem \ref{obstruction}), which measures the
obstruction to lifting the endomorphism $\varphi$ to a $W_2(k)$-algebra
endomorphism of $\An(W_2(k))$, is closely related to the matrix
$J_\gamma^T - J_\gamma$ associated with the solutions of the differential
equations in Corollary \ref{constants}. These matrices are linked by the
identities
\begin{align*}
    J_\varphi \, \omega^{-1} \, J_\varphi^T & = \omega^{-1} + C, \\ 
    \hat J_\varphi^T \, \omega \, \hat J_\varphi & = \omega + J_\gamma^T - J_\gamma, \\
    J_\varphi & = \hat J_\varphi^{[p]},
\end{align*}
derived in Theorem \ref{theoremidjac} and Corollary \ref{matrix-id}.
Therefore, in general and in view of the discussion in
Section \ref{poissonmorp}, the matrix $C$ vanishes if and only if
$J_\gamma$ is symmetric. Tsuchimoto’s result corresponds to the special
case where $J_\gamma$ is the zero matrix. In general, however, we obtain
the following criterion.

\begin{proposition}
\label{diffeqpois}
For $i=1,\dots,2n$, let $f_i \in Z = k[x_1,\dots,x_{2n}]$ be the unique
solutions of the differential equation
$$
\pdv[p-1]{ }{x_i}
\Biggl(
\sum_{l=1}^{n}
\pdv{\varphi(x_l)}{x_i}\, \varphi(x_{n+l})
\Biggr)
=
f^p +
\pdv[p-1]{f}{x_i}.
$$
Then $\varphi$ admits a lifting to a $W_2(k)$-algebra endomorphism of
$\An(W_2(k))$ if and only if
$$
\pdv{f_i}{x_j} = \pdv{f_j}{x_i}, \quad i,j=1,\dots,2n.
$$
\end{proposition}

\begin{proof}
We begin by observing that $f_i = \gamma_i^p$ for $i=1,\dots,2n$.
Thus the matrix
$$
J = \bigl(\pdv{f_i}{x_j}\bigr) \in \Mat_{2n}(Z)
$$
is related to $J_\gamma$ by
$$
J = J_\gamma^{[p]} = \Bigl(\pdv{\gamma_i}{y_j}^p\Bigr).
$$
It follows that $J$ is symmetric if and only if $J_\gamma$ is symmetric.
By the remarks preceding the proposition, this completes the proof.
\end{proof}

Notice that, with the notation $J = (\pdv{f_i}{x_j})$ as in the proof above, we have
the identities
\begin{align*}
    J_\varphi \, \omega^{-1} \, J_\varphi^T & = \omega^{-1} + C, \\ 
    J_\varphi^T \, \omega \, J_\varphi & = \omega + J^T - J.
\end{align*}

 \section{Applications and concluding remarks}\label{lastsect}

In \cite[page 2328]{Bavula2008Inversion} it was conjectured that any endomorphism $\varphi$ of
$\An(k)$ is injective. This question is equivalent to the injectivity
of the induced map $\varphi_Z$ on the center.
Tsuchimoto \cite[Lemma 17]{Tsuchimoto2003Preliminaries} proved injectivity for $n=1$, 
but later Makar-Limanov \cite[page 793]{MakarLimanov2012}
gave an example of a non-injective endomorphism of $A_2(k)$ of degree
$p^2+p-1$.
Observe that any flat (and hence also étale) endomorphism of $Z$ is
injective, since $Z$ is an integral domain. Thus Theorem~\ref{etalecenter} 
establishes the injectivity of $\varphi$ 
under the stated degree bounds, of which $\deg(\varphi) < p$ is a special case.

In \cite[Theorem 2.7, Theorem 4.3]{LT}, we proved that endomorphisms of Weyl algebras over fields of
characteristic zero are flat and that birational endomorphisms are automorphisms.
These results crucially depend on the fact that reductions of endomorphisms
of Weyl algebras to (large) positive characteristic induce étale endomorphisms on the
center. Theorem \ref{etalecenter} gives explicit bounds on the characteristic
that guarantee this property. In particular, Theorem 2.7 (flatness) and Theorem 4.3 (birationality implies automorphism) in \cite{LT} hold
for endomorphisms $\varphi$ of Weyl algebras over fields of characteristic $p > \deg{\varphi}$.

\end{document}